\documentclass[12pt]{amsart}
\usepackage{amssymb,amsmath,amstext,graphicx,float}
\usepackage{amsfonts,amsthm,verbatim,subfig}
\usepackage[all]{xy} 
\ifx\pdftexversion\undefined 
\usepackage[a4paper,colorlinks,link
=black,filecolor=black,citecolor=black,urlcolor=black,pdfstartview=FitH]{hyperref}
\else

\usepackage[a4paper,colorlinks,linkcolor=black,filecolor=black,citecolor=black,urlcolor=black,pdfstartview=FitH]{hyperref}
\fi

\font\sixbb=msbm6
\font\eightbb=msbm8
\font\twelvebb=msbm10 scaled 1095
\newfam\bbfam
\textfont\bbfam=\twelvebb \scriptfont\bbfam=\eightbb
                           \scriptscriptfont\bbfam=\sixbb

\newtheorem{theorem}{Theorem}[section]

\newtheorem{lemma}[theorem]{Lemma}
\newtheorem{observation}[theorem]{Observation}

\newcommand{\F}{\mathcal{F}}
\newcommand{\C}{\mathcal{C}}

\newcommand{\R}{\mathbb{R}}

\newcommand{\dist}{\text{dist}}

\begin{document}
\thispagestyle{empty}
\author{Shiliang Gao}
\address{Shiliang Gao: Department of Mathematics,
University of Michigan, Ann Arbor} \email{shiliang@umich.edu}
\author{Shira Zerbib}\address{Shira Zerbib: Department of Mathematics,
University of Michigan, Ann Arbor} \email{zerbib@umich.edu}

\begin{abstract}
A family of sets satisfies the $(p,q)$ property if among every $p$ members of it some $q$ intersect. Given a number $0<r\le 1$, a set $S\subset \R^2$ is called $r$-fat if there exists a point $c\in S$ such that $B(c,r) \subseteq S\subseteq B(c,1)$, where $B(c,r)\subset \R^2$ is a disk of radius $r$ with center-point $c$. We prove constant upper bounds $C=C(r)$ on the piercing numbers in families 
of $r$-fat sets in $\R^2$ that satisfy the $(2,2)$ or the $(4,3)$ properties. This extends results by Danzer \cite{Danzer} and Karasev \cite{karasev} on the piercing numbers in intersecting families of disks in the plane, as well as a result by Kyn\v{c}l and Tancer \cite{KT} on the piercing numbers in families of units disks in the plane satisfying the $(4,3)$ property.     
\end{abstract}

\title{The $(2,2)$ and $(4,3)$ properties in families of fat sets in the plane}
\maketitle

\section{introduction}

\subsection{The $(p,q)$ problem.}
A classical theorem of Helly \cite{helly} asserts that if $\F$ is a family of convex sets in $\R^d$ in which every $d+1$ members intersect, then all the members in $\F$ intersect.    
Helly's theorem initiated the broad area of research in discrete geometry, dealing with questions regarding the number of points needed to pierce families of convex sets in $\R^d$ satisfying certain intersection properties.   

Given integers $p\ge q >1$, a family $\F$ of sets is said to satisfy the {\em $(p,q)$ property} if among any $p$ elements of $\F$ there exist $q$ elements with a non-empty intersection.  
We denote by $\tau(\F)$ the {\em piercing number} (also called in the literature {\em covering number}, {\em stubbing number}, or {\em hitting number}) of $\F$, namely the minimum size of a set of points in $\R^d$ intersecting every set in $\F$. 
The {\em matching number} of $\F$, namely the maximum number of pairwise disjoint sets in $\F$, is denoted by $\nu(F)$. Clearly, $\nu(\F) \le \tau(\F)$. 

If $\nu(\F)=1$ then we say that $\F$ is an {\em intersecting} family of sets. Note that $\F$ satisfies the $(p,2)$ property if and only if $\nu(\F)\le p-1$, and in particular, $\F$ is intersecting if and only if it satisfies the $(2,2)$ property.     

Helly's theorem is that if a family $\F$ of convex sets in $\R^d$ satisfies the $(d+1,d+1)$ property, then $\tau(\F)=1$. Finding the piercing numbers of families of sets in $\R^d$ satisfying the $(p,q)$ property has been referred in the literature as the {\em $(p,q)$ problem}. 

In 1992 Alon and Kleitman \cite{AK} resolved a long standing conjecture of Hadwiger and Debrunner \cite{HD}, proving that for every $p\ge q\ge d+1$ there exists a constant $c=c(d;p,q)$ such that if a family $\F$ of convex sets in $\R^d$ satisfies $(p,q)$ property then $\tau(\F)\le c$.  

Alon and Kleitman's proof is constructive, however, the upper bounds on $c(d;p,q)$ given by their proof are far from being optimal. For example, the Alon-Kleitman proof gives $c(2;4,3)\le 253$; however, in \cite{432} Kleitman, Gy\'arf\'as and T\'oth proved that at most $13$ points are needed to pierce a family of convex sets in $\R^2$ satisfying the $(4,3)$ property. Over the last few decades extensive research has been done to improve the Alon-Kleitman bounds, see e.g., \cite{KST, 432, KT} and the references therein. For an excellent survey on the $(p,q)$ problem we refer the reader to \cite{Eckhoff}.

\subsection{The $(2,2)$ and $(4,3)$ properties in $\R^2$}
There does not exist a general constant upper bound on the piercing number $\tau(\F)$ when $\F$ is an intersecting family of convex sets in $\R^2$, as is exemplified by a family of lines in the plane in general position, in which the piercing number is at least half the number of lines. However, in some cases, when $\F$ consists of   
certain ``nice" sets, constant upper bounds on the piercing numbers in intersecting families can be proved. One such example is the following result by Danzer \cite{Danzer} on intersecting families of disks.
\begin{theorem}[Danzer, \cite{Danzer}]\label{danzer22}
An intersecting family of disks in $\R^2$ has $\tau(\F)\le 4$.
\end{theorem}

An extension of this result for families of homothets in the plane was proved by Karasev \cite{karasev}. Given a centrally symmetric body $B \subset \R^2$ and a number $0< r \le 1$, a {\em $(B,r)$-homothet} is a set $tB+u$ for some number $r \le t \le 1$ and a point $u\in \R^2$. Karasev proved the following: 

\begin{theorem}[Karasev, \cite{karasev}]\label{karasev22}
Let $B$ be a centrally symmetric body in $\R^2$. If $\F$ is an intersecting family of $(B,\frac{1}{2})$-homothets then $\tau(\F)\le 3$.
\end{theorem}
  
As mentioned above, for families satisfying the $(4,3)$ property we have the following.
\begin{theorem}[Kleitman--Gy\'arf\'as--T\'oth, \cite{432}]\label{KGT}
If $\F$ is a family of convex sets in the plane that satisfies the $(4,3)$ property, then $\tau(\F)\le 13$.
\end{theorem}   
  
It seems that improving the bound in Theorem \ref{KGT} in general is a hard task.
However, bounds on the piercing numbers in families satisfying the $(4,3)$ property can be significantly improved if one considers only certain restricted such families. For example, Kyn\^cl and Tancer proved in \cite{KT} a tight upper bound  for families of unit disks satisfying the $(4,3)$ property:  
\begin{theorem}[Kyn\^cl--Tancer, \cite{KT}]\label{kt} 
If $\F$ is a family of unit disks in the plane which satisfies the $(4,3)$ property then $\tau(\F)\le 3$, and there exist  families of sets achieving this bound.
\end{theorem} 
Other types of families $\F$ satisfying the $(4,3)$ property that were proved in \cite{KT} to achieve $\tau(\F)\le 3$ are families of translations of a triangle in $\R^2$ and families of line segments in $\R^d$.

Furthermore, Theorems \ref{danzer22} and \ref{karasev22} imply the following bounds for families of disks or homothetes: 
\begin{theorem}\label{danzer}
If $\F$ is a family of disks in $\R^2$ satisfying the $(4,3)$ property then $\tau(\F)\le 5$.
\end{theorem} 
\begin{theorem}\label{karasev} 
If $\F$ is a family of $(B,\frac{1}{2})$-homothets of a centrally symmetric body $B$ in $\R^2$ and $\F$ satisfies the $(4,3)$ property, then $\tau(\F)\le 4$. 
\end{theorem}

Both theorems follow from a simple observation:
\begin{observation}\label{observation}
Let $\C$ be a collection of convex sets in $\R^2$ and let $t\ge3$. If for every family $\F\subset \C$ that satisfies the $(2,2)$ property we have $\tau(\F)\le t$, then for every family $\F\subset \C$ that satisfy the $(4,3)$ property we have $\tau(\F)\le t+1$.
\end{observation}
\begin{proof}
Let $\F\subset \C$ be a family of sets satisfying the $(4,3)$ property. If $|F| \le 3$ the observation is trivial. Thus we may assume $\F$ contains at least $4$ sets. Observe that $\nu(\F)\le 2$, for otherwise
 a matching of size $3$ together with any other set in $\F$ is a collection of four sets violating the $(4,3)$ property. If $\nu(\F)=1$ then $\F$ satisfies the $(2,2)$ property, and thus $\tau(\F)\le t$. Suppose $\nu(\F)=2$ and  let $A,B$ be two disjoint sets in $\F$. We claim that either every  set in $\F\setminus \{A,B\}$ intersects $A$ or every set in $\F\setminus \{A,B\}$ intersects $B$; indeed, if there exist $D,E \in \F\setminus \{A,B\}$ such that $D\cap A = E\cap B=\emptyset$, then the foursome $A,B,D,E$ violates the $(4,3)$ property. Assume without loss of generality that every set in $\F\setminus \{A,B\}$ intersects $B$. Then $\F= \F_B \cup \F_{AB} \cup \{A\}$, where $\F_B\subset \F$ is the family of sets in $\F$ intersecting $B$ and not intersecting $A$, and $\F_{AB}\subset \F$ is the family of sets in $\F$ intersecting both $A$ and $B$. Observe that $\F_B$ must satisfy the $(3,3)$ property, for otherwise a non-intersecting triple of sets in $\F_B$ together with $A$ constitute a foursome of sets violating the $(4,3)$ property. Hence by Helly's theorem $\tau(\F_B)=1$. Furthermore, $F_{AB}\cup \{A\}$ is an intersecting family of sets,  since if $D,E \in \F_{AB}$ are disjoint then the sets $A,B,D,E$ violate the $(4,3)$ property. Thus we have $\tau(\F) \le  \tau(F_{AB}\cup\{A\})+\tau(\F_B)\le t+1$, proving the observation.         
\end{proof}

\subsection{Our results}
In this note we further investigate the piercing numbers in families of sets in $\R^2$ satisfying the $(2,2)$ or the $(4,3)$ properties.   
To this end we define the notion of {\em fat sets}, and prove upper bounds on the piercing numbers in families of fat sets in the plane satisfying these intersection properties. 

Given a number $0<r\le 1$, a set $S\subset \R^2$ will be called {\em $r$-fat} if there exists a point $c\in S$ such that $B(c,r) \subseteq S\subseteq B(c,1)$, where $B(c,r)$ is the disk in $\R^2$ of radius $r$ with center-point $c$. 
Thus a $1$-fat set is a unit disk, and every disk of radius $d\le 1$ is an $r$-fat set for every $r\le d$. 
Note that for $r<1$, an $r$-fat set is not necessarily convex. 

We extend the results mentioned above by establishing constant upper bounds, depending only on $r$, on the piercing numbers in intersecting families of $r$-fat sets in the plane, or in families of convex $r$-fat sets in the plane that satisfy the $(4,3)$ property. We consider only such families whose union is compact. We prove the following.  
 
\begin{theorem}\label{main1}
Let $\F$ be a family of (not necessarily convex) $r$-fat sets in the plane satisfying the $(2,2)$ property.  
\begin{enumerate} 
\item If $\sqrt{8}-2\le r\le 1$ then $\tau(\F) \le 4$.
\item If $0.68\le r \le 1$ then $\tau(\F) \le 5$.
\item For every $0<r\le 1$ we have $\tau(\F)\le (\lceil \frac{\sqrt{2}}{r} \rceil)^2$, and in particular, if $r\ge 0.5$ then $\tau(\F) \le 9$.
\end{enumerate}
\end{theorem}

\begin{theorem}\label{main2}
Let $\F$ be a family of convex $r$-fat sets in the plane satisfying the $(4,3)$ property.  
\begin{enumerate} 
\item If $\sqrt{8}-2\le r\le 1$ then $\tau(\F) \le 4$.
\item If $0.68\le r\le 1$ then $\tau(\F) \le 5$.
\item For every $0<r\le 1$ we have $\tau(\F)\le (\lceil \frac{\sqrt{2}}{r} \rceil)^2 + 1$, and in particular, if $r\ge 0.5$ then $\tau(\F) \le 10$.
\end{enumerate}
\end{theorem}

Our proof methods rely on bounding the piercing numbers $\tau(\F)$ by the number of disks of radius $r$ needed to cover certain bounded regions in the plane.  
In Section 2 we establish preliminary lemmas needed for the proofs of Theorems \ref{main1} and \ref{main2}, and the proofs are then given in Section 3.

\section{Seven lemmas} 

Let $\dist$ denote the Euclidean distance in $\R^2$.  

\begin{lemma}\label{lemmageometry}
Let $r> 0$, and suppose that $\mathcal{D}$ is a family of disks in $\R^2$, each of them of radius at least $r$. Let $C(\mathcal{D})$ be the set of all center-points of disks in $\mathcal{D}$. If there exists a point $c\in \R^2$ such that $C(\mathcal{D})\subset B(c,r)$, then $\bigcap \mathcal{D} \neq \emptyset$. 
\end{lemma}
\begin{proof}
We have $\dist(c,p) \le r$ for every $p\in C(\mathcal{D})$, implying $c\in D$ for every $D\in \mathcal{D}$. 
\end{proof}

For an $r$-fat set $S \subset\R^2$ let $c_S\in S$ be a point in $\R^2$ for which $B(c_S,r) \subseteq S\subseteq B(c_S,1)$.

Let $\F$ be a family of $r$-fat sets in $\R^2$ satisfying the conditions of Theorems \ref{main1} or \ref{main2}.
We may assume that 
$|\F| \ge 4$, for otherwise both theorems are trivial. 

Fix $A, B \in \F$ such that $\dist(c_A,c_B) = \max_{F,F'\in \F}\dist(c_F,c_{F'})$, and  write $d=\dist(c_A,c_B)$. By rotating and translating $\F$ we may assume that $c_B$ is the point $(\sqrt{8},0)$, and $c_A$ is to the left of $c_B$, that is, $c_A$ is the point $(\sqrt{8}-d,0)$. We make this choice in order to simplify the computations in the sequel. 

\begin{lemma}\label{distatmost2}
For every $D, E \in \F\setminus \{A,B\}$ we have $\dist(c_D,c_E) \le 2$.
\end{lemma}
\begin{proof}
The lemma is trivial if $d\le 2$. If $d>2$ then $\F$ is non-intersecting, thus it must satisfy the $(4,3)$ property, and moreover, $A \cap B = \emptyset$. If in addition $D\cap E = \emptyset$ for some $D, E \in \F\setminus \{A,B\}$, then the  sets $A,B,D,E$ violate the $(4,3)$ property. Therefore $D,E$ intersect, implying $\dist(c_D,c_E) \le 2$. 
\end{proof}

By the same arguments as in the proof of Observation \ref{observation} we have: 
\begin{lemma}\label{lemma1}
If $A,B \in \F$ are disjoint then $\nu(\F)=2$. Moreover, either $A$ intersects every disk in $\F\setminus \{B\}$ or $B$ intersects every disk in $\F\setminus \{A\}$.
\end{lemma}

From now on we will assume without loss of generality that $B$ intersects every set in $\F\setminus \{A\}$. Thus if $\F$ is non-intersecting then $A,B$ are disjoint and $\F = F_B \cup F_{AB} \cup \{A\},$ where $F_B$ and $F_{AB}$ are defined as before.

For a subfamily $\F' \subseteq \F$ let $C(\F') = \{c_F \mid F\in \F'\}$. 
For any two real numbers $a\ge b$ let $H(a,b)=\{(x,y)\in \R^2 \mid b\le y \le a\}$. Define 
\[ L(d) = \begin{cases} \Big(B(c_A,d) \cap B(c_B,2)\Big) \cup \{c_A\} & d \le \sqrt{8} \\ B(c_A,2) \cap B(c_B,2) & d > \sqrt{8} \end{cases} .\] 

By the maximality of $d$ we have $C(\F)\subset B(c_A,d)$. Moreover, by our assumption, $B$ intersects every set in $\F\setminus \{A\}$, and therefore we have also  $C(\F\setminus \{A\})\subset B(c_B,2)$. Thus if  
 $d\le \sqrt{8}$ then $C(\F)\subset L(d)$,  and if $d>\sqrt{8}$, then $C(\F_{AB})\subset L(d)$ by definition.

Let $m$ denote the maximum $y$-coordinate of a point in $C(\F)$ in the case $d\le \sqrt{8}$, or in $C(\F_{AB})$ otherwise. By computing the maximum of the function $y$ in the domains $$\{(d,x,y)\mid (x-(\sqrt{8}-d))^2+y^2\le d^2, ~(\sqrt{8}-x)^2+y^2\le 4, ~0\le d\le \sqrt{8} \},$$ or $$\{(d,x,y)\mid (x-(\sqrt{8}-d))^2+y^2\le 4, ~(\sqrt{8}-x)^2+y^2\le 4, ~d\ge 0\}, $$ we obtain  

\begin{lemma}\label{maxheight}
 $m \le \sqrt{3.5}$.
\end{lemma}

The following three lemmas will be the key tool in our proofs:
\begin{lemma}\label{lemma2}
Suppose that the set $L(d) \cap H(m,m-2)$ is contained in the union of at most $k$ disks of radius $r$. then \[ \tau(\F) \le \begin{cases} k & d \le \sqrt{8} \\ k+2 & d>\sqrt{8} \end{cases} \]
\end{lemma}
\begin{proof}
If $d\le \sqrt{8}$, then Lemma \ref{distatmost2} implies $C(\F)\subset H(m,m-2)$, and by the arguments above,   $C(\F)\subset L(d) \cap H(m,m-2)$.    
Therefore the conditions of the lemma imply that there exist $k$ disks of radius $r$ containing $C(\F)$ in their union. By Lemma \ref{lemmageometry} this means that $\F$ can be pierced by the center-points of these $k$ disks. 

If $d > \sqrt{8}$, then $A,B$ are disjoint, and as in the proof of Observation \ref{observation}, $F_B$ is a family of convex sets satisfying the $(3,3)$ property, which entails by Helly's theorem, $\tau(\F_B)\le 1$. Furthermore, Lemma \ref{distatmost2} implies  $C(\F_{AB}) \subset L(d) \cap H(m,m-2)$, and thus $C(\F_{AB})$ is contained in the union of $k$ disks of radius $r$. It  
follows from Lemma \ref{lemmageometry} that $\tau(\F_{AB})\le k$, implying $\tau(\F)\le \tau(\F_{AB}) + \tau(F_B) +\tau(\{A\}) \le k+2$.    
\end{proof}

Define a (possibly empty) rectangle $R\subset \R^2$ as follows.  
\[ R = \begin{cases} \big\{(x,y)\mid \sqrt{8}-2 \le x \le \sqrt{8},~ -1 \le y \le 0  \big\} & d \le \sqrt{8} \\ \big\{(x,y)\mid \sqrt{8}-2\le x\le \sqrt{8}-d+2,~ -1\le y\le 0 \big\} & d > \sqrt{8} \end{cases} .\]

\begin{lemma}\label{lemma5}
If the set $\big((L(d) \setminus\{c_A\}) \cap H(1,-1)\big) \cup R$ is contained in the union of $k$ disks of radius $r$, then so is $(L(d) \setminus\{c_A\}) \cap H(a,a-2)$, for any real number $a$.
\end{lemma}
\begin{proof}
We first claim that it is enough to prove the lemma for 
$1\le a\le \sqrt{3.5}$. Indeed, if $a< 1$ we reflect the set $(L(d)\setminus\{c_A\}) \cap H(a,a-2)$ about the $x$-axis, and prove the lemma for the reflected set, and if $a > \sqrt{3.5}$ then by Lemma \ref{maxheight} we have $$(L(d) \setminus\{c_A\}) \cap H(a,a-2) \subseteq (L(d) \setminus\{c_A\}) \cap H(\sqrt{3.5},\sqrt{3.5}-2),$$ proving the claim. 

 So let $1\le a\le \sqrt{3.5}$. We claim that if $\big((L(d) \setminus\{c_A\}) \cap H(1,-1)\big) \cup R $ is contained in $\bigcup_{i=1}^{k} B((x_i,y_i),r),$ then $(L(d) \setminus\{c_A\}) \cap H(a,a-2)$ is contained in $\bigcup_{i=1}^{k} B((x_i,y_i+a-1),r).$   

To prove this, we show that if 
$(x,y) \in (L(d) \setminus\{c_A\}) \cap H(a,a-2)$ then
\begin{equation} 
(x,y-(a-1)) \in \big((L(d)\setminus\{c_A\}) \cap H(1,-1)\big) \cup R. 
\end{equation}
To see that (1) holds, note that if $a-2 \le y \le a-1$ then  $(x,y-(a-1)) \in R$, and  
 for $a-1 < y \le a$ we have both $$(x-(\sqrt{8}-d))^2+(y-(a-1))^2 < (x-(\sqrt{8}-d))^2+y^2 \le d^2$$ and $$(\sqrt{8}-x)^2+(y-(a-1))^2 < (\sqrt{8}-x)^2+y^2 \le 4,$$  
implying $(x,y-(a-1))\in L(d)\setminus \{c_A\}$. 
Moreover, $y \le a$ entails $y-(a-1)\le 1$, and therefore (1) is true.
\end{proof}

Let $R'=R_1\cup R_2$ be a union of rectangles, where  
$$R_1 = \big\{(x,y)\mid \sqrt{8}-d \le x \le \sqrt{8}-2, ~1-\sqrt{3.5} \le y \le 0  \big\} \text{ and }$$  
$$R_2=\big\{(x,y)\mid \sqrt{8}-2 \le x \le \sqrt{8}, ~-1 \le y \le 0  \big\}.$$

Applying similar arguments to those in the proof of Lemma \ref{lemma5} we obtain:

\begin{lemma}\label{lemma6}
For $d \le \sqrt{8}$, if $\big(L(d) \cap H(1,-1)\big) \cup R'$ is contained in the union of $k$ disks of radius $r$, then so is $L(d) \cap H(a,a-2)$, for any real number $a$. 
\end{lemma}

\section{Proof of Theorems \ref{main1} and \ref{main2}}
For any number $a\in \R$ let $H^+(a)=\{(x,y)\in \R^2 \mid y\ge a\}$ and $H^-(a)=\{(x,y)\in \R^2 \mid y\le a\}$. Note that for two real numbers $a\ge b$ we have $H(a,b) = H^-(a) \cap H^+(b)$. 

Let $A,~B , ~d, ~L(d), ~m, ~R, ~R'$ be defined as in the previous section. If $A,B$ intersect then $d\le 2$, and thus $C(\F)\subset B(c_A,2)\cap B(c_B,2)$. If $A,B$ are disjoint then, as before, we assume that $B$ intersects every set in $\F\setminus \{A,B\}$, and therefore, $C(\F\setminus \{A,B\})\subset B(c_A,d)\cap B(c_B,2)$.

\subsection{The case $\sqrt{8}-2 \le r \le 1$.}
we distinct three subcases. 
\medskip  

{\bf Case 3.1.1.} $d \le \sqrt{8}$ and there exists  $F\in \F$ such that $c_F \in H^+(1.1)$.
In this case, by Lemma \ref{distatmost2}, $C(\F\setminus \{A,B\}) \subset H^+(-0.9)$. This implies $$C(\F)\subseteq  L(d) \cap H^+(-0.9).$$
Observe that $L(d) \cap H^+(-0.9)$ is  contained in the union of four disks 
$\bigcup_{i=1}^4 B(p_i,\sqrt{8}-2)$, where $p_1=((\sqrt{8}-2)\cos (0.24\pi),(\sqrt{8}-2)\sin (0.24\pi))$, 
$p_2=(2.01,1.053)$, $p_3=(2.4972,-0.115)$ and $p_4=(1.64,-0.33)$  
(see Figure 1). 
Lemma \ref{lemmageometry} thus implies $\tau(\F)\le 4$. 
\bigskip 
\begin{figure}
\centering
\includegraphics[scale=0.53]{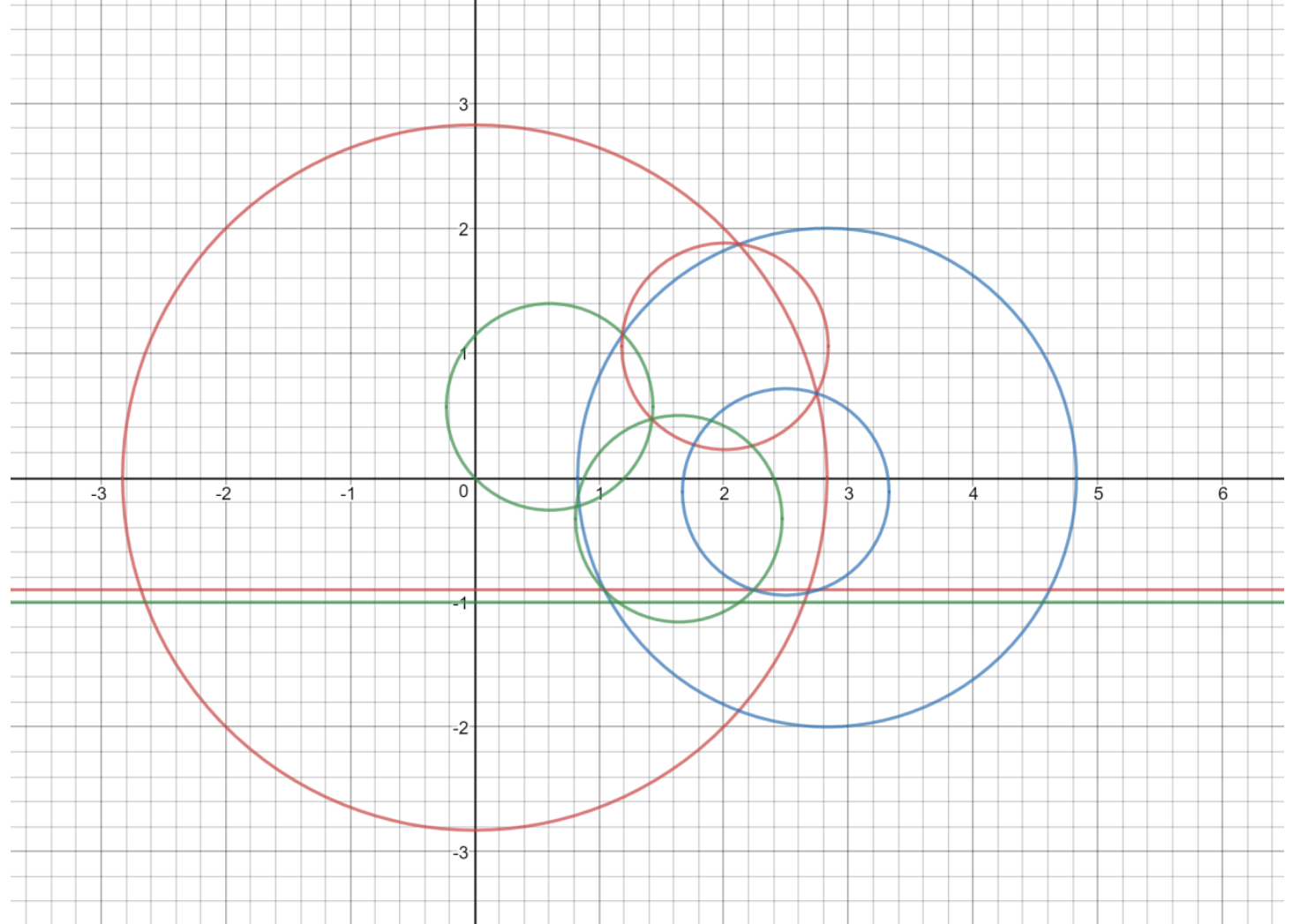}
\caption{$L(d)\cap H^+(-0.9)$ is contained in the union of four disks of radius $\sqrt{8}-2$ in Case 3.1.1.}
\end{figure}

{\bf Case 3.1.2.} $d \le \sqrt{8}$ and $C(\F) \subset H^-(1.1)$.
In this case  we have  $$C(\F)\subset L(d)\cap H^-(1.1)\subset \bigcup_{i=1}^4 B(q_i,\sqrt{8}-2),$$ where  $q_1$ is the point $((\sqrt{8}-2)\cos(0.24\pi),(\sqrt{8}-2)\sin (0.24\pi))$,  $q_2=(1.5739,-0.6133)$, $q_3=(2.5357,-0.204)$, and  $q_4=(1.95,0.7)$ (see Figure 2). Thus again, by Lemma \ref{lemmageometry} we have  $\tau(\F)\le 4$, and in particular Assertion (1) in Theorem \ref{main1} is proved.
 \begin{figure}
\centering
\includegraphics[scale=0.43]{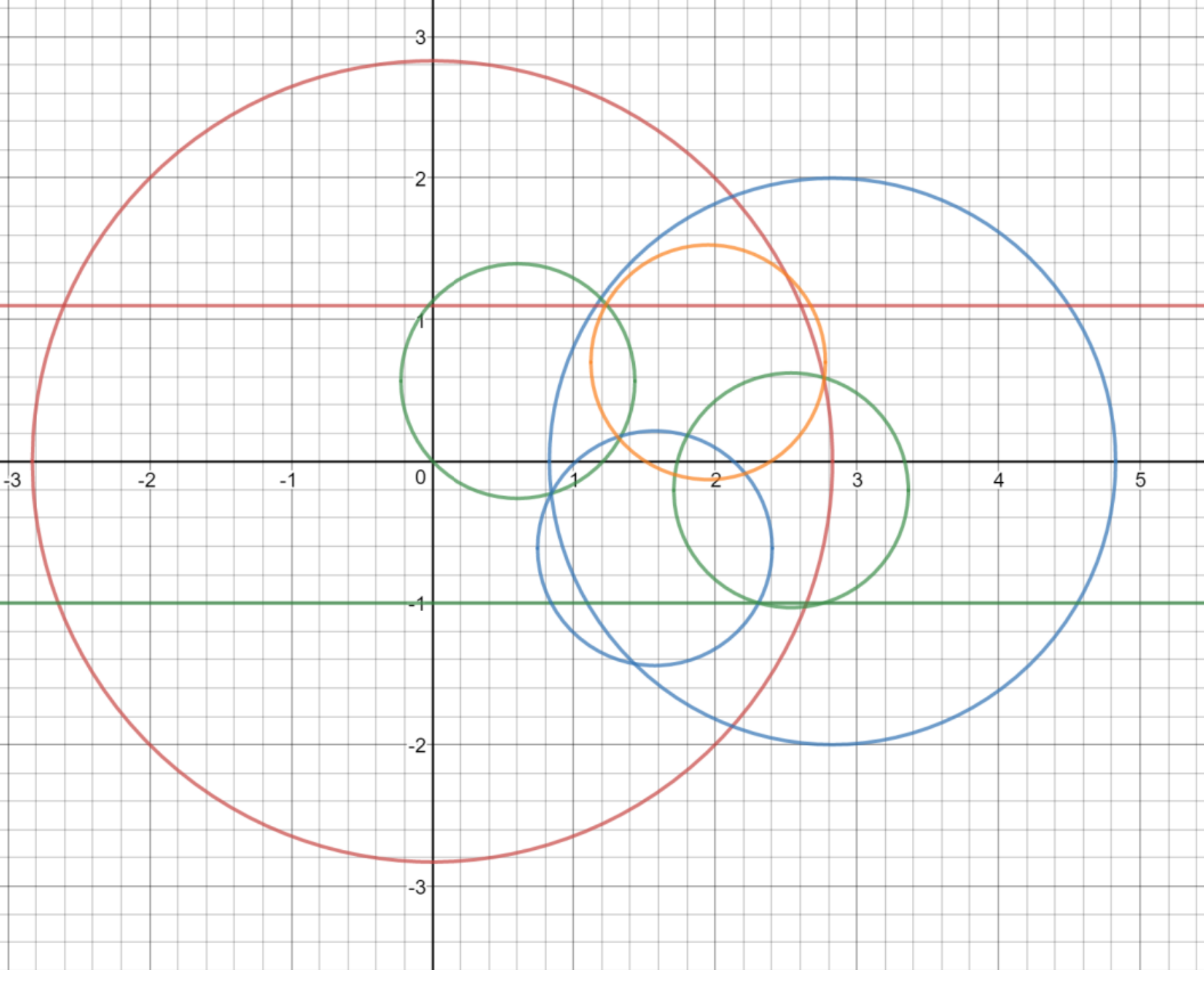}
\caption{$L(d)\cap H^-(1.1)$ is contained in the union of four disks of radius $\sqrt{8}-2$ in Case 3.1.2.}
\end{figure}
\medskip

{\bf Case 3.1.3.} $d > \sqrt{8}$. In this case
$A, B$ are disjoint, and we assume that $B$ intersects every set in $\F\setminus \{A,B\}$. Let $\F_B$ and $\F_{AB}$ as before. Then we have $\tau(\F_B \cup \{A\})\le 2$.
Furthermore, we observe that $C(\F_{AB}) \subset L(d)$ and 
$$
 L(d) \subset B((\sqrt{2},2-\sqrt{2}),\sqrt{8}-2)\cup B((\sqrt{2},\sqrt{2}-2),\sqrt{8}-2)
$$ (see Figure 3). Therefore, by Lemma \ref{lemmageometry} we have $\tau(\F_{AB})\le 2$, implying $\tau(\F)\le 4$. 
This completes the proof of Assertion (1) in Theorem \ref{main2}.
\begin{figure}
\centering
\includegraphics[scale=0.68]{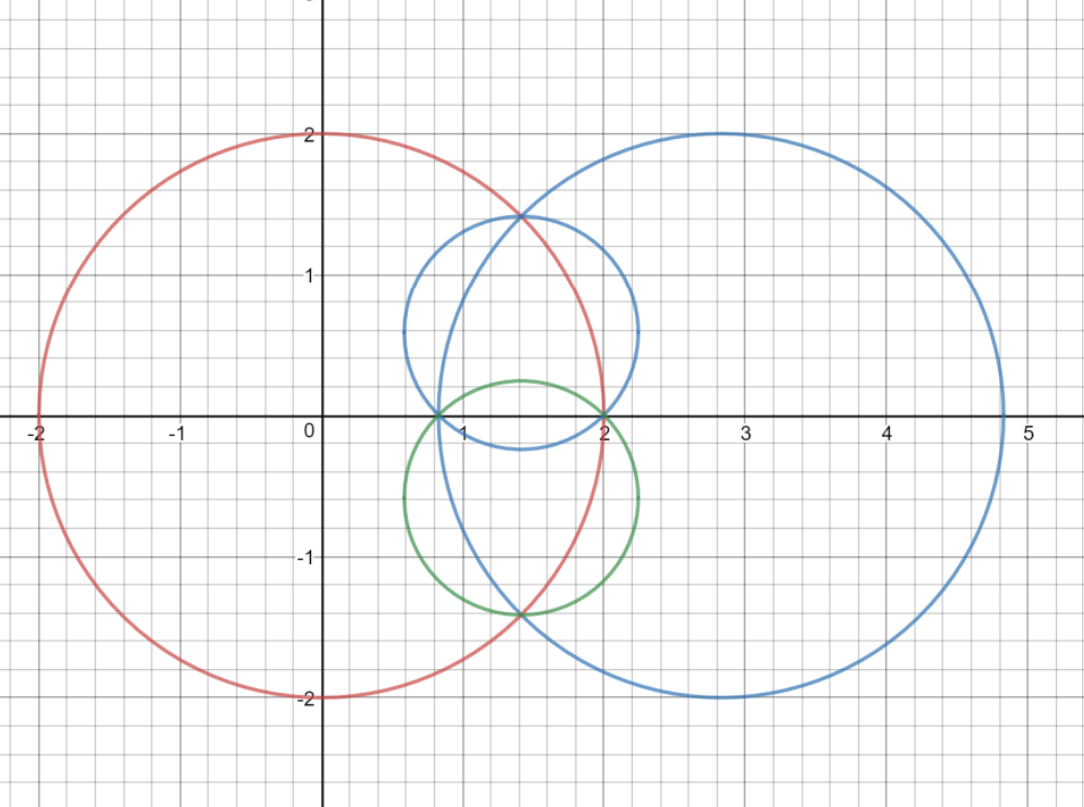}
\caption{$L(d)$ is contained in the union of two disks of radius $\sqrt{8}-2$ in Case 3.1.3.}
\end{figure}

\subsection{The case $0.68\le r< \sqrt{8}-2$.}
Here we consider two subcases. 
\medskip

{\bf Case 3.2.1.} $d \le \sqrt{8}$. In this case, by Lemma \ref{lemma2}, in order to prove Assertions (2) in Theorems \ref{main1} and \ref{main2} it suffices to show that the set $L(d) \cap H(m,m-2)$ is contained in the union of five disks of radius $0.68$. 
 Moreover, by Lemma \ref{lemma6}, this will follow if show that the set $\big(L(d) \cap H(1,-1)\big)\cup R'$ is contained in the union of five disks of radius $0.68$. Now observe that the union of disks of radius $0.68$ with center-points 
$p_1 = (0.49,-0.465), ~p_2 = (1.477, 0.6262), ~p_3 = (2.445,-0.456),$ $p_4 = (1.435, 0.435)$ and  $p_5 = (2.3162, 0.55)$, contains $\big(L(d) \cap H(1,-1)\big)\cup R'$ (see Figure 4). In particular Assertion (2) in Theorem \ref{main1} is proved.
\begin{figure}
\centering
\includegraphics[scale=0.29]{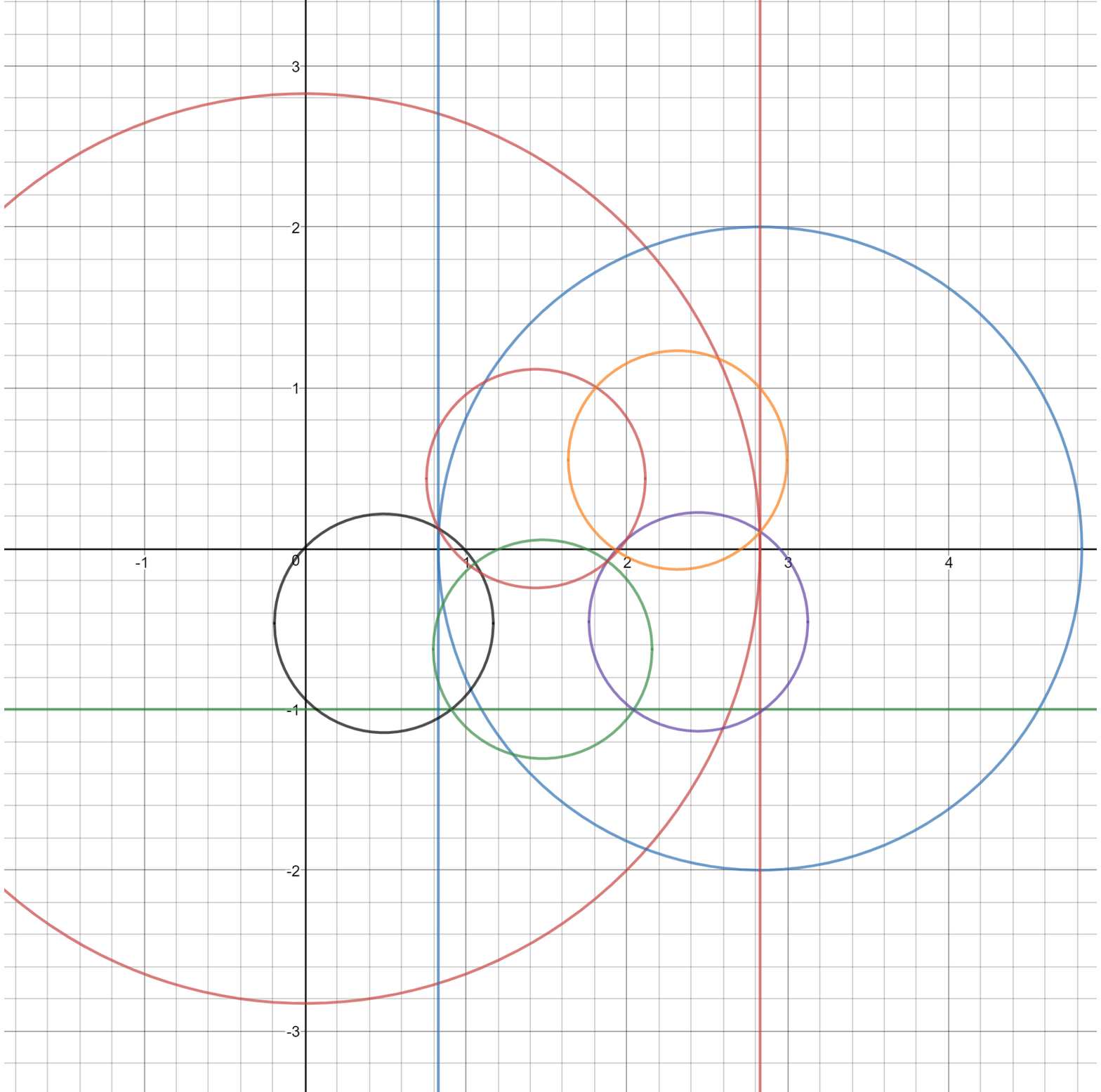}
\caption{$L(d) \cap H(1,-1)$ is contained in the union of five disks of radius $0.68$ in Case 3.2.1.}
\end{figure}
\medskip

{\bf Case 3.2.2.} $d > \sqrt{8}$. 
By Lemma \ref{lemma2},  it is enough to show that $L(d) \cap H(m,m-2)$ is contained in the union of $3$ disks of radius $0.68$. Since $L(d) = L(d) \setminus \{c_A\}$ when $d > \sqrt{8}$, this will follow from Lemma \ref{lemma5}
if we show that the set  
$\big((L(d)\setminus \{c_A\}) \cap H(1,-1)\big)\cup R$ is contained in the union of $3$ disks of radius $0.68$. Indeed, the set $\big((L(d)\setminus \{c_A\}) \cap H(1,-1)\big)\cup R$ is contained in  $\bigcup_{i=1}^3 B(q_i,0.68)$, where $q_1 = (\sqrt{2},0.43), q_2 = (1.1, -0.5),$ and $q_3 = (1.7, -0.5)$ (see Figure 5). This concludes the proof of Assertion (2) in Theorem \ref{main2}.
\begin{figure}
\centering
\includegraphics[scale=0.3]{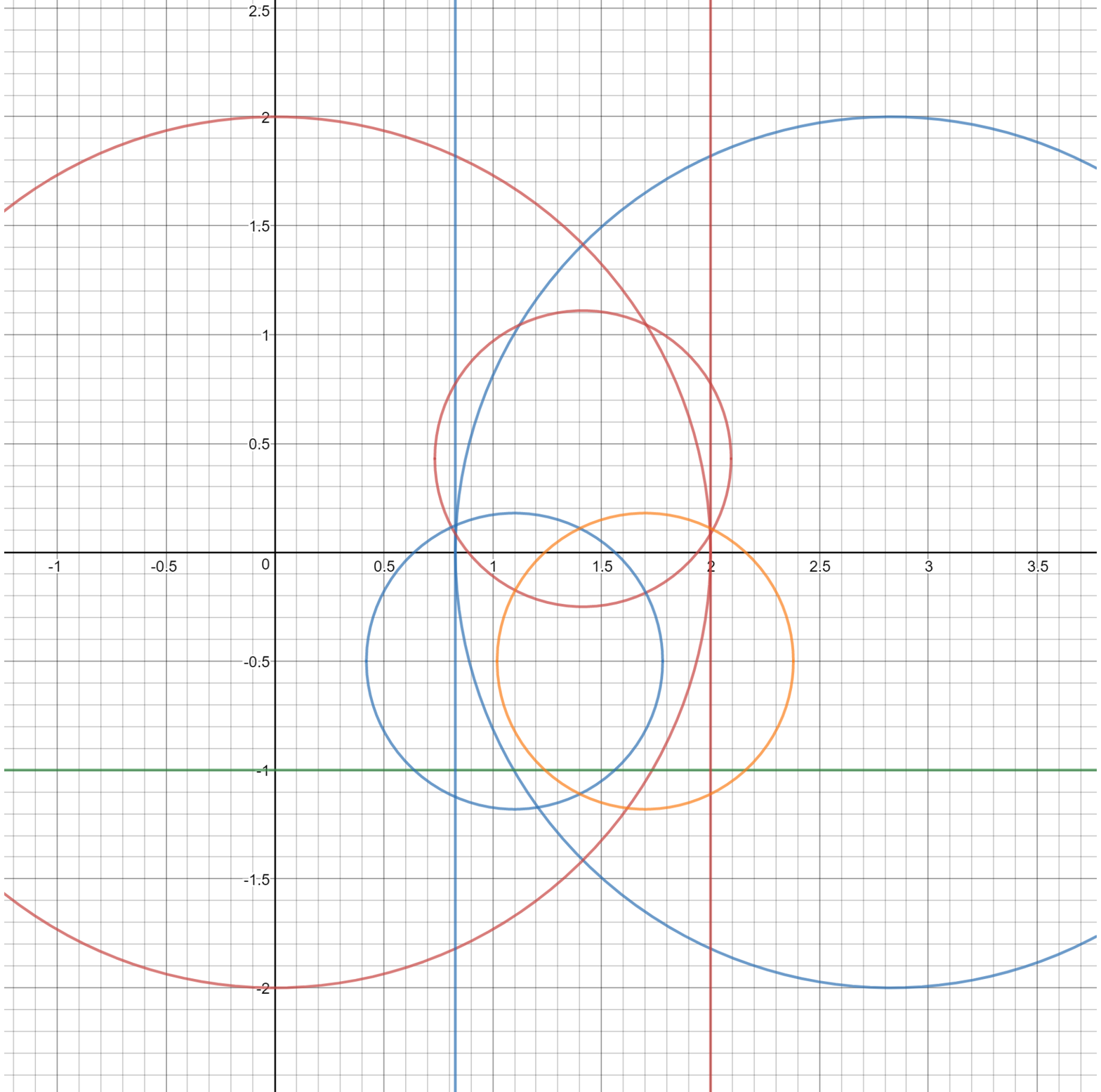}
\caption{$L(d) \cap H(1,-1)$ is contained in the union of three disks of radius $0.68$ in Case 3.2.2.}
\end{figure}
\medskip

\subsection{Proof of the third assertions in Theorems \ref{main1} and \ref{main2}}
Suppose first that $\F$ satisfies the $(2,2)$ property. Then $d \le 2$, and by Lemmas  \ref{lemma2} and \ref{lemma6} it suffices to show that $\big(L(d)\cap H(1,-1)\big)\cup R'$ is contained in the union of  $(\lceil \frac{\sqrt{2}}{r} \rceil)^2$ disks of radius $r$. 

Note that $\big((L(d))\cap H(1,-1)\big)\cup R'$ is contained in the rectangle 
$\{(x,y)\mid \sqrt{8}-2 \le x \le \sqrt{8}, -1\le y \le 1\}$, and this rectangle can be covered by $(\lceil \frac{\sqrt{2}}{r}\rceil)^2$ squares of edge length $\sqrt{2}r$. To conclude the proof of Theorem \ref{main1}, observe that the union of $k$ squares with edge length $\sqrt{2}r$ is contained in the union of $k$ disks of radius $r$. 

If $\F$ is a family of convex $r$-fat sets satisfying the $(4,3)$ property, then we apply Assertion (3) in Theorem \ref{main1} together with Observation \ref{observation}, to conclude $\tau(\F) \le (\lceil \frac{\sqrt{2}}{r}\rceil)^2+1$. Thus Theorem \ref{main2} is proved.

\section*{Acknowledgment}
This research was done as part of an REU project at the University of Michigan, Ann Arbor, during the Summer Semester, 2017.   
We are grateful to the Department of Mathematics at the University of Michigan for supporting this project. The work on this paper  was completed while the second author was in residence at the Mathematical Sciences Research Institute in Berkeley, California, during the Fall 2017 semester, and the material is based upon work supported by the National Science Foundation under Grant No. DMS-1440140.

\end{document}